\documentclass[11pt,reqno]{amsart}
\usepackage[left=1in,right=1in,top=1in,bottom=1in]{geometry}
\usepackage[T1]{fontenc}
\usepackage{lmodern}
\usepackage{microtype}
\usepackage{mathtools}
\usepackage{amssymb}
\usepackage{mathrsfs}
\usepackage{xcolor}
\usepackage[colorlinks=true,linkcolor=blue,citecolor=blue,urlcolor=blue]{hyperref}
\usepackage[noabbrev]{cleveref}

\numberwithin{equation}{section}
\newtheorem{theorem}{Theorem}[section]
\newtheorem{proposition}[theorem]{Proposition}
\newtheorem{lemma}[theorem]{Lemma}
\newtheorem{corollary}[theorem]{Corollary}
\newtheorem*{blackbox}{Miller's black box}
\theoremstyle{remark}
\newtheorem{remark}[theorem]{Remark}
\theoremstyle{definition}

\theoremstyle{plain}

\crefname{theorem}{Theorem}{Theorems}
\Crefname{theorem}{Theorem}{Theorems}
\crefname{proposition}{Proposition}{Propositions}
\Crefname{proposition}{Proposition}{Propositions}
\crefname{lemma}{Lemma}{Lemmas}
\Crefname{lemma}{Lemma}{Lemmas}
\crefname{corollary}{Corollary}{Corollaries}
\Crefname{corollary}{Corollary}{Corollaries}
\crefname{remark}{Remark}{Remarks}
\Crefname{remark}{Remark}{Remarks}
\crefname{definition}{Definition}{Definitions}
\Crefname{definition}{Definition}{Definitions}

\newcommand{\R}{\mathbb R}
\newcommand{\Z}{\mathbb Z}
\newcommand{\C}{\mathbb C}
\newcommand{\eps}{\varepsilon}
\newcommand{\Om}{\Omega}
\newcommand{\pa}{\partial}
\newcommand{\cL}{\mathcal L}
\newcommand{\cA}{\mathcal A}
\newcommand{\cE}{\mathcal E}
\newcommand{\one}{\mathbf 1}
\newcommand{\ip}[2]{\left\langle #1,#2\right\rangle}
\newcommand{\norm}[2][]{\left\lVert #2\right\rVert_{#1}}
\newcommand{\dd}{\,\mathrm d}
\newcommand{\dist}{\operatorname{dist}}
\newcommand{\Ran}{\operatorname{Ran}}
\newcommand{\Spec}{\operatorname{Spec}}
\newcommand{\Id}{\mathrm I}

\newcommand{\rev}[1]{{\color{black}#1}}

\title[The $\eps^{-1}$ eigenvalue cutoff for periodic waves]
{Uniform boundary observability for periodic waves, I:\\
up to the $\boldsymbol{\eps^{-1}}$ eigenvalue scale}

\date{}


\author{Jun Geng}
\address{School of Mathematics and Statistics, Lanzhou University, Lanzhou 730000, Gansu, P.R. China.}
\email{gengjun@lzu.edu.cn}

\author{Shumao Wang}
\address{School of Mathematical Sciences, Capital Normal University, Beijing 100048, P.R. China.}
\email{shumaowang@cnu.edu.cn}

\author{Wei Wang}
\address{School of Mathematical Sciences, Peking University, Beijing 100871, P.R. China}
\email{wwmath166@outlook.com}
\email{2201110024@stu.pku.edu.cn}

\begin{document}

\begin{abstract}
Let $\cL_\eps=-\operatorname{div}(A(x/\eps)\nabla)$ be a scalar
divergence-form operator on a bounded $C^3$ domain, where $A$ is symmetric,
uniformly elliptic, Lipschitz continuous, and periodic.  For initial data
spectrally truncated at the eigenvalue level $N$, we establish whole-boundary
observability and projected exact controllability, uniformly in the range
$N\leq\delta\eps^{-1}$.  This extends the cutoff
$N\lesssim\eps^{-1/2}$ obtained by Lin and Shen
\cite{LinShen2022,LinShenErratum2022}.  The key input is a two-sided boundary
estimate for fixed-width spectral packets of $\sqrt{\cL_\eps}$ at frequencies
that satisfy $\eps\mu^2\ll1$.  Its lower bound combines quantitative
Dirichlet corrector estimates with a forcing-stable microscopic boundary
Cauchy estimate.  Miller's spectral packet criterion and the Hilbert
uniqueness method then convert this packet estimate into the corresponding
wave results, with constants independent of $\eps$ and $N$.
\end{abstract}

\maketitle

\section{Introduction}

Let $\Om\subset\R^d$ $ (d\geq 2) $ be a bounded $C^3$ domain, and consider the operator
\begin{equation}\label{eq:operator}
  \cL_\eps u_\eps=-\pa_i\big(a_{ij}(x/\eps)\pa_j u_\eps\big),\qquad 0<\eps<1.
\end{equation}
We impose homogeneous Dirichlet boundary conditions and sum over repeated
indices.  Throughout, we assume that $A=(a_{ij})$ is real, symmetric,
$\Z^d$-periodic, and, for some
$0<\kappa<1$ and $M>0$,
\begin{equation}\label{eq:ellipticity}
 \kappa|\xi|^2\leq A(y)\xi\cdot\xi\leq \kappa^{-1}|\xi|^2,
 \qquad |A(y)-A(z)|\leq M|y-z|,\qquad \xi,y,z\in\R^d.
\end{equation}
At scales much larger than $\eps$, the operator \eqref{eq:operator} is
governed by a constant-coefficient effective operator $\cL_0$.  This classical
homogenization principle
\cite{BensoussanLionsPapanicolaou1978,JikovKozlovOleinik1994,AvellanedaLin1987}
describes each fixed spectral scale, but it does not determine how far the
admissible spectral range may extend as $\eps\to0$, while the estimates remain
uniform.  For a modern account of quantitative periodic homogenization,
including uniform boundary regularity and spectral convergence, see
\cite{Shen2018}.  For complementary operator-norm estimates for periodic
hyperbolic equations in $\R^d$, obtained by a spectral approach, see
\cite{DorodnyiSuslina2018}.
Boundary observation makes this distinction unavoidable, since it detects
both macroscopic wave propagation and the interaction between the wavelength
and the oscillating medium.  We therefore ask the following quantitative
question: \emph{how far into the spectrum can periodic waves be observed with
constants independent of $\eps$?}

Let $\{\lambda_{\eps,k}\}_{k\geq1}$ be the Dirichlet spectrum of
$\cL_\eps$, repeated with multiplicity, and let
$\{\phi_{\eps,k}\}_{k\geq1}$ be an orthonormal eigenbasis of $L^2(\Om)$.
For $N>0$, set
\begin{equation}\label{eq:AN}
 \cA_{\eps,N}=\Ran\one_{(0,N]}(\cL_\eps)
 =\left\{\sum_{\lambda_{\eps,k}\leq N}c_k\phi_{\eps,k}:c_k\in\C\right\}.
\end{equation}

We study solutions of
\begin{equation}\label{eq:wave}
 \begin{cases}
  \pa_t^2u_\eps+\cL_\eps u_\eps=0&\text{in }\Om\times(0,T),\\
  u_\eps=0&\text{on }\pa\Om\times(0,T),\\
  (u_\eps,\pa_tu_\eps)|_{t=0}=(u_{\eps,0},u_{\eps,1}).
 \end{cases}
\end{equation}
Its conserved energy is
\begin{equation}\label{eq:energy}
 \cE_\eps(u_{\eps,0},u_{\eps,1})
 =\int_\Om\big(A(x/\eps)\nabla u_{\eps,0}\cdot
                    \overline{\nabla u_{\eps,0}}
                   +|u_{\eps,1}|^2\big)\dd x.
\end{equation}

For a fixed coefficient field, the Hilbert uniqueness method of Lions
identifies exact boundary controllability with an observability inequality for the adjoint wave; see \cite{Lions1988}. The geometric theory of Bardos, Lebeau, and
Rauch gives the corresponding sharp condition for smooth wave operators
\cite{BardosLebeauRauch1992}. Here, the geometry is fixed, and the additional
question is uniformity: can the observability constants remain bounded as
$\eps\to0$?  We impose the spectral restriction
$\cA_{\eps,N}\times\cA_{\eps,N}$ and allow $N=N(\eps)$ to increase.  The
parameter $N$ is an \emph{eigenvalue} cutoff; the associated wave-frequency
cutoff is $\sqrt N$. Consequently, $\mu$ will always denote the wave frequency,
so $N=\mu^2$ at the spectral level.

The low-frequency restriction is not merely technical.  Nonuniform boundary
observation in rapidly oscillating media was already exhibited by
Avellaneda, Bardos, and Rauch \cite{AvellanedaBardosRauch1992}.
Concentration mechanisms for rapidly oscillating rough coefficients were
developed further by Castro and Zuazua
\cite{CastroZuazua2002concentration,CastroZuazua2007addendum}; these examples
lie outside the periodic Lipschitz class considered here, but they underscore
the genuinely high-frequency nature of obstruction.

For the one-dimensional wave equation with periodic density, Castro and
Castro--Zuazua obtained a detailed description of the low- and high-frequency
regimes \cite{Castro1999,CastroZuazua2000high,CastroZuazua2000low}.  In that
model, the critical wavelength is comparable to the period, and hence the
critical eigenvalue scale is $\lambda\simeq\eps^{-2}$.  In higher dimensions,
Avellaneda and Lin studied weak limits of boundary controls through the
homogenization of the Poisson kernel \cite{AvellanedaLin1989}, while Lebeau
proved low-frequency observability for a wave equation with oscillating
density \cite{Lebeau2000}.  These results reveal the appearance of microscopic
high-frequency effects. Still, the corresponding critical scales do not transfer directly
to \eqref{eq:operator}: in dimensions $d\geq2$, an oscillating density is not
equivalent to an oscillation in the principal part. See the discussions in
\cite{KenigLinShen2013,LinShen2022}.
A closely related loss of uniform observability occurs in space
discretizations of wave equations, where filtering high frequencies is again
essential; see the survey \cite{Zuazua2005}.

For the principal-part operator \eqref{eq:operator}, the important
elliptic information is due to Kenig, Lin, and Shen, who obtained quantitative
bounds for eigenfunctions and unit spectral packets
\cite{KenigLinShen2013}.  In particular, if $\phi_\eps$ is an $L^2$-normalized
Dirichlet eigenfunction with eigenvalue $\lambda_\eps$, their estimates give
the upper bound
\begin{equation}\label{eq:eigenupperintro}
 \int_{\pa\Om}|\nabla\phi_\eps|^2\dd\sigma
 \leq C\lambda_\eps(1+\eps\lambda_\eps)
 \qquad\text{when }\eps^2\lambda_\eps\leq1,
\end{equation}
and, in the scalar case, the matching lower bound
\begin{equation}\label{eq:eigenlowerintro}
 \int_{\pa\Om}|\nabla\phi_\eps|^2\dd\sigma
 \geq c\lambda_\eps
 \qquad\text{when }\eps\lambda_\eps\leq\delta.
\end{equation}
The same work also proves an upper estimate for unit spectral packets.
Taken together, \eqref{eq:eigenupperintro}--\eqref{eq:eigenlowerintro}
identify $\lambda\simeq\eps^{-1}$ as the largest scale on which the available
higher-dimensional elliptic theory gives two-sided boundary control.  For
related normal-derivative bounds in the non-oscillatory setting, see
\cite{HassellTao2002}.

At the hyperbolic level, Lin and Shen
\cite{LinShen2022,LinShenErratum2022} compared the oscillatory wave with a
first-order Dirichlet-corrected homogenized wave.  They then used oscillatory
Rellich identities to convert interior estimates of the comparison remainder into
control of its boundary trace.  In the notation used here, their estimates
\cite{LinShen2022,LinShenErratum2022} yield
\begin{equation}\label{eq:oldcutoff}
                         N\lesssim T^{-1}\eps^{-1/2}.
\end{equation}
See \cite[Theorem~1.1]{LinShen2022} and
\cite[Theorem~0.1]{LinShenErratum2022}.  For a fixed observation time,
\eqref{eq:oldcutoff} is the scale $N\lesssim\eps^{-1/2}$.  On the other hand, Lin-Shen proposed the following natural question in the paper \cite{LinShen2022},

\emph{"These results suggest that one may be able to extend Theorem 1.1 to the case $N\leq C\eps^{-1}$. But this remains unknown."}

The above is \textbf{the main motivation} of our paper. However, a direct use of \cite[Theorem~1.4]{KenigLinShen2013} in a wave expansion does
not close the gap between these two scales.  A wave does not
encounter its eigenfunctions one at a time: nearby frequencies interfere,
their boundary traces are not orthogonal, and in several dimensions no
uniform spectral gap separates them.  What is missing is therefore a lower
bound for entire fixed-width spectral packets that remains stable under the
quasimode remainder carried by such superpositions.  We establish this packet
estimate and thereby attain the scale suggested by the elliptic theory in
\cite{KenigLinShen2013}.

\begin{theorem}
\label{thm:main}
Let $\Om$ be a bounded $C^3$ domain and assume \eqref{eq:ellipticity}.
There exist $\eps_0,\delta,T_0>0$ and $c,C>0$, depending only on
$d,\kappa,M$, and $\Om$, such that the following holds.  If
\begin{equation}\label{eq:newcutoff}
       0<\eps<\eps_0,\qquad 0<N\leq\delta\eps^{-1},\qquad T\geq T_0,
\end{equation}
and the initial data in \eqref{eq:wave} belong to
$\cA_{\eps,N}\times\cA_{\eps,N}$, then
\begin{equation}\label{eq:observability}
 c\,\cE_\eps(u_{\eps,0},u_{\eps,1})
 \leq \frac1T\int_0^T\!\int_{\pa\Om}|\nabla u_\eps|^2\dd\sigma\dd t
 \leq C\,\cE_\eps(u_{\eps,0},u_{\eps,1}).
\end{equation}
The constants in \eqref{eq:observability} are independent of
$\eps,N$, and $T$ in the range \eqref{eq:newcutoff}.
\end{theorem}

\begin{remark}\label[remark]{rem:comparisonLS}
For fixed $T$, \cref{thm:main} improves the cutoff
$N\lesssim\eps^{-1/2}$ in \cite{LinShen2022,LinShenErratum2022} to
$N\lesssim\eps^{-1}$; once $T\geq T_0$, the new cutoff is also independent of
$T$.  We obtain the improved cutoff by first establishing a boundary lower bound for spectral packets in fixed-width frequency intervals, then applying Miller’s criterion to deduce wave observability. This approach avoids tracking a homogenization remainder over the observation interval.
\end{remark}

\begin{remark}\label[remark]{rem:scope}
Theorem~\ref{thm:main} refers to scalar equations with observation on the whole
boundary.  Systems and observations on a proper subset of $\pa\Om$ are not
treated here and would require new lower bounds of the spectral package.
\end{remark}

\begin{remark}\label[remark]{rem:openrange}
One-dimensional oscillating-density models admit
uniform estimates up to the scale $\eps^{-2}$, but are not equivalent to
\eqref{eq:operator} for
$d\geq2$.  For the present principal-part problem, the range
\[
                         \eps^{-1}<N\lesssim\eps^{-2}
\]
remains open.  Reaching it requires new boundary information and does not follow by iterating the first-order corrector argument, which will be proposed in a forthcoming paper.
\end{remark}

By the Hilbert uniqueness method~\cite{Lions1988}, \cref{thm:main} gives the
corresponding control statement.  Write
$P_{\eps,N}=\one_{(0,N]}(\cL_\eps)$.
On $H^{-1}(\Om)$ this notation denotes the dual spectral projection
\[
 P_{\eps,N}F
 =\sum_{\lambda_{\eps,k}\leq N}
   \ip{F}{\phi_{\eps,k}}_{H^{-1},H^1_0}\,\phi_{\eps,k}.
\]

\begin{corollary}\label[corollary]{cor:control}
Under the assumptions of \cref{thm:main}, for every
$(\theta_{\eps,0},\theta_{\eps,1})\in L^2(\Om)\times H^{-1}(\Om)$ there is a
Dirichlet control $g_\eps\in L^2(\pa\Om\times(0,T))$ such that the weak solution
of
\[
 \pa_t^2v_\eps+\cL_\eps v_\eps=0,\qquad
 v_\eps|_{\pa\Om}=g_\eps,\qquad
 (v_\eps,\pa_tv_\eps)|_{t=0}=(\theta_{\eps,0},\theta_{\eps,1})
\]
satisfies
\[
 P_{\eps,N}v_\eps(T)=P_{\eps,N}\pa_tv_\eps(T)=0.
\]
The minimal-norm control obeys two-sided bounds, uniform in $\eps$ and $N$,
\begin{equation}\label{eq:controlnorm}
 \begin{aligned}
 c_T\big(\norm[L^2]{P_{\eps,N}\theta_{\eps,0}}
          +\norm[H^{-1}]{P_{\eps,N}\theta_{\eps,1}}\big)
 &\leq \norm[L^2(\pa\Om\times(0,T))]{g_\eps}\\
 &\leq C_T\big(\norm[L^2]{P_{\eps,N}\theta_{\eps,0}}
          +\norm[H^{-1}]{P_{\eps,N}\theta_{\eps,1}}\big).
 \end{aligned}
\end{equation}
For $T\geq T_0$, the two constants can be chosen of the order $T^{-1/2}$,
as dictated by the normalization in \eqref{eq:observability}.
\end{corollary}

We now outline the main point of the proof for the observability estimate. Set $S_\eps=\sqrt{\cL_\eps}$ and, for $r>0$,
$\Om_r=\{x\in\Om:\dist(x,\pa\Om)<r\}$.  \rev{For a fixed, sufficiently small packet width $\gamma>0$, we establish
the following stationary boundary estimate for spectral packets of $S_\eps$:}
\begin{equation}\label{eq:packetintro}
 c\mu^2\norm[L^2(\Om)]{f}^2
 \leq\norm[L^2(\pa\Om)]{\nabla f}^2
 \leq C\mu^2\norm[L^2(\Om)]{f}^2
\end{equation}
whenever
\[
 f\in\Ran\one_{[(\mu-\gamma)_+,\mu+\gamma]}(\sqrt{\cL_\eps}),
 \qquad \eps\mu^2\leq\delta.
\]
\rev{Each function in this packet space satisfies the quasimode estimate}
\begin{equation}\label{eq:packetresidualintro}
  (\cL_\eps-\mu^2)f=q,
  \qquad \norm[L^2(\Om)]{q}\lesssim\gamma\mu\norm[L^2(\Om)]{f}.
\end{equation}
The upper bound follows from \cite[Theorem~4.3]{KenigLinShen2013}; the new
assertion is the lower bound on the entire packet space.  This distinction is
essential: estimate \eqref{eq:eigenlowerintro} concerns one exact frequency,
while \eqref{eq:packetintro} must remain coercive under arbitrary
interference among all frequencies in a certain interval.  Neither orthogonality of
boundary traces nor a uniform spectral gap is available to recover this
coercivity by summing the eigenfunction estimates.

The remainder in \eqref{eq:packetresidualintro} is the price of passing from a
mode to a packet and cannot be discarded.  Its control is global, and on an
individual boundary cylinder of radius $\eps$, the forcing may still be large
compared to the local gradient energy.  The Cauchy mechanism behind
\cite[Lemma~5.3]{KenigLinShen2013} must therefore be made stable under forcing
and then coupled to the global packet bound.

The proof achieves this through a single energy transfer.  The
Dirichlet corrector and a constant-coefficient Rellich identity first place a
definite amount of gradient energy in an $\eps$-collar.  A microscopic
boundary Cauchy estimate can then be used on those $\eps$-cylinders where the
packet energy does not
degenerate and the local forcing is perturbative.  The good–bad decomposition allows us to combine this local principle with the global estimate
\eqref{eq:packetresidualintro}: bounded overlap shows that cylinders 
failing either condition carry only an absorbable portion of the collar
energy.  The resulting mechanism is summarized by the uniform chain
\begin{equation}\label{eq:packetmechanismintro}
 \begin{aligned}
  \norm[L^2(\pa\Om)]{\nabla f}^2
  &\geq \frac c\eps\int_{\Om_{\rho\eps}}|\nabla f|^2\dd x\\
  &\geq \mu^2\big(c-C\eps\mu^2-C\gamma\big)
          \norm[L^2(\Om)]{f}^2.
 \end{aligned}
\end{equation}
The first line of \eqref{eq:packetmechanismintro} is the new
remainder-stable passage from microscopic collar mass to boundary mass; the
second is the corrector--Rellich estimate.  Their combination converts global
quasimode control into boundary coercivity without resolving the packet into
individual eigenfunctions.

Now, the improved cutoff follows from the two error terms in \eqref{eq:packetmechanismintro}.  We first fix $\gamma$ so that the packet error is absorbed and
then impose $\eps\mu^2\leq\delta$ to absorb the homogenization error.  Since
$N=\mu^2$, this is exactly the range $N\lesssim\eps^{-1}$.  Once \eqref{eq:packetintro} is established, Miller's criterion \cite{Miller2012} yields fixed-time wave observability from this
spectral package estimate, and time translation together with HUM gives the control statement.  The new exponent therefore does not arise from a longer asymptotic expansion. Still, from assigning
the first-order corrector a different role: it creates stationary collar
coercivity rather than tracking a corrected wave for time $T$. The boundary
lower bound is established at the packet level, and the wave evolution is
reconstructed only afterward.

The paper is organized as follows.  In \cref{sec:elliptic} we collect 
Dirichlet corrector estimates and upper bounds for spectral packets.  In
\cref{sec:lowerpacket} we prove the stable microscopic Cauchy estimate, the
boundary-layer mass estimate, and the two-sided packet theorem.  In
\cref{sec:observability} we verify the abstract admissibility hypotheses,
apply Miller's criterion, and carry out the projected HUM construction.

\section{Elliptic inputs and spectral packets}\label{sec:elliptic}

\rev{Unless otherwise indicated, $c$ and $C$ denote positive constants
depending only on $d$, $\kappa$, $M$, and $\Om$.
Their values may vary from line to line.}
The homogenized operator is denoted by
$\cL_0=-\operatorname{div}(\widehat A\nabla)$.

For $1\leq j\leq d$, the Dirichlet corrector $\Phi_{\eps,j}$ is the solution of
\begin{equation}\label{eq:DirCorrector}
 \cL_\eps\Phi_{\eps,j}=0\quad\text{in }\Om,
 \qquad \Phi_{\eps,j}=x_j\quad\text{on }\pa\Om.
\end{equation}
\rev{We write $\Phi_\eps=(\Phi_{\eps,1},\ldots,\Phi_{\eps,d})$ and
use the notation
\begin{equation}\label{eq:gradientmatrixconvention}
 \nabla\Phi_\eps
 =\left(\frac{\partial\Phi_{\eps,j}}{\partial x_i}\right)_{1\leq i,j\leq d}.
\end{equation}}

\rev{The following proposition collects the estimates for Dirichlet
correctors that will be used below.}

\begin{proposition}
\label[proposition]{prop:corrector}
Let $F\in L^2(\Om)$, and let $u_\eps,u_0\in H^1_0(\Om)$ be the weak solutions of
\[
                    \cL_\eps u_\eps=F=\cL_0u_0.
\]
Then
\begin{align}
 \norm[H^1_0(\Om)]{u_\eps-u_0-(\Phi_\eps-x)\cdot\nabla u_0}
     &\leq C\eps\norm[L^2(\Om)]{F},\label{eq:H1corrector}\\
 \norm[L^\infty(\Om)]{\Phi_\eps-x}&\leq C\eps.\label{eq:Phibound}
\end{align}
Moreover, there are $c_0,\eps_0>0$ such that
\begin{equation}\label{eq:Phiinvertible}
             |(\nabla\Phi_\eps(x))\xi|\geq c_0|\xi|
 \quad\text{for a.e. }x\in\Om_{c_0\eps},\quad 0<\eps<\eps_0.
\end{equation}
\end{proposition}

The estimate \eqref{eq:H1corrector} is formula~(2.8) in
\cite[Theorem~2.4]{KenigLinShen2013}. The estimate \eqref{eq:Phibound} is
formula~(2.5) in \cite[Lemma~2.2]{KenigLinShen2013}.
We also use the uniform Dirichlet corrector Lipschitz estimate
$\norm[L^\infty(\Om)]{\nabla\Phi_\eps}\leq C$ invoked in the proof of
\cite[Theorem~4.3]{KenigLinShen2013}.
In the scalar case, \cite[Lemma~5.1]{KenigLinShen2013}, together with the
uniform boundary Lipschitz bound for $\Phi_\eps$, gives the pointwise matrix
estimate \cite[formula~(5.5)]{KenigLinShen2013}; this is precisely
\eqref{eq:Phiinvertible}.  The thickness of this boundary layer is a fixed
structural multiple of $\eps$.  Under the present coefficient and boundary
regularity assumptions,
the first two corrector estimates also hold for systems; the scalar assumption
is essential for \eqref{eq:Phiinvertible}.

For $\mu>0$ and $0<\gamma\leq1$, define the
packet space
\begin{equation}\label{eq:packetspace}
 \mathscr P_\eps(\mu,\gamma)
 =\Ran\one_{[(\mu-\gamma)_+,\mu+\gamma]}(S_\eps),
 \qquad (a)_+:=\max\{a,0\}.
\end{equation}
\rev{Let $f\in\mathscr P_\eps(\mu,\gamma)$ with $\norm[L^2]{f}=1$, and set}
\begin{equation}\label{eq:quasimode}
              q=(\cL_\eps-\mu^2)f\rev{.}
\end{equation}
\rev{If $\mu\geq2\gamma$, the spectral theorem gives}
\begin{equation}\label{eq:packetbasic}
 \norm[L^2]{q}\leq3\gamma\mu,\qquad
 c\mu\leq\norm[L^2]{\nabla f}\leq C\mu,\qquad
 \norm[L^2]{\cL_\eps f}\leq C\mu^2.
\end{equation}
Indeed, on the spectral support of $f$ we have
$|\sigma^2-\mu^2|\leq\gamma(2\mu+\gamma)\leq3\gamma\mu$ and
$\mu/2\leq\sigma\leq3\mu/2$.  The last two assertions then follow from the
spectral theorem, the energy identity
$\ip{\cL_\eps f}{f}=\int_\Om A(x/\eps)\nabla f\cdot\overline{\nabla f}$,
and the uniform ellipticity.

\rev{We next record the upper boundary estimate for spectral packets.}

\begin{proposition}
\label[proposition]{prop:upperpacket}
There is $C>0$ such that, if $0<\gamma\leq1$, $\mu\geq1$, and
$f\in\mathscr P_\eps(\mu,\gamma)$, then
\begin{equation}\label{eq:upperpacket}
 \norm[L^2(\pa\Om)]{\nabla f}^2
 \leq C\mu^2(1+\eps\mu^2)\norm[L^2(\Om)]{f}^2
\end{equation}
whenever $\eps^2\mu^2\leq c$.  In the bounded range $0<\mu<1$, one instead
has $\norm[L^2(\pa\Om)]{\nabla f}^2\leq
C\norm[L^2(\Om)]{f}^2$.
\end{proposition}

\begin{proof}
We first reduce to real-valued functions.  Because the coefficients are real,
the relevant spectral projections commute with complex conjugation; applying
the real-valued estimates to the real and imaginary parts and adding the
resulting squared inequalities yields the complex-valued versions.
In the notation of \cite[Theorem~4.3]{KenigLinShen2013}, a unit packet is
$[\sqrt\lambda,\sqrt\lambda+1)$, where $\lambda$ is an eigenvalue parameter
and $\lambda\simeq\mu^2$; the required estimate is formula~(4.11) there.  Since
$0<\gamma\leq1$, the interval in \eqref{eq:packetspace} is covered by at most
three disjoint unit frequency intervals (assigning every shared endpoint to
one interval).  A possible piece below frequency $1$ is kept separate and is
handled by the bounded-frequency estimate below.  Decompose $f$ into the
corresponding orthogonal pieces $f=\sum_{\ell=1}^m f_\ell$, where $m\leq3$.
The finite-sum inequality and the unit-packet estimate give
\[
 \norm[L^2(\pa\Om)]{\nabla f}^2
 \leq m\sum_{\ell=1}^m\norm[L^2(\pa\Om)]{\nabla f_\ell}^2
 \leq C\mu^2(1+\eps\mu^2)\sum_{\ell=1}^m\norm[L^2]{f_\ell}^2,
\]
which is \eqref{eq:upperpacket}.
For the bounded-frequency assertion, the min-max principle and uniform
ellipticity bound the dimension of the spectral subspace below frequency
$2$ by a structural constant.  Applying the single-eigenfunction bound
\cite[Theorem~1.3]{KenigLinShen2013} to an orthonormal basis gives a uniform
$O(1)$ bound.  In the alternative branch $\eps^2\lambda\geq1$, the restriction
$\lambda\leq4$ itself implies $\eps^{-1}\leq\sqrt\lambda\leq2$.
The stated estimate now follows from Cauchy--Schwarz.
\end{proof}

\rev{The proof of the lower bound also uses the following boundary-layer estimate.}

\begin{corollary}
\label[corollary]{cor:upperstrip}
Fix $R>0$.  There are $c>0$ and $C_R>0$, depending only on the structural
parameters and on $R$, such that, for every $0<\gamma\leq1$, $\mu>0$, and
$f\in\mathscr P_\eps(\mu,\gamma)$ satisfying $\eps\mu^2\leq c$, one has
\begin{equation}\label{eq:upperstrip}
 \int_{\Om_{R\eps}}|\nabla f|^2\dd x
 \leq C_R\eps(1+\mu)^2\norm[L^2(\Om)]{f}^2.
\end{equation}
\end{corollary}

\begin{proof}
First, suppose that the packet is supported above the frequency $1$.  For a single
unit packet,
the proof of estimates (4.16)--(4.17) in
\cite[Theorem~4.3]{KenigLinShen2013}, with the collar thickness replaced by
$R\eps$, gives
\[
 \frac1\eps\int_{\Om_{R\eps}}|\nabla f|^2\dd x
 \leq C_R(1+\mu)^2\bigl(1+\eps(1+\mu)^2\bigr)
       \norm[L^2(\Om)]{f}^2.
\]
The interval defining $\mathscr P_\eps(\mu,\gamma)$ is the union of at most
three such unit packets.  Decomposing $f$ into the corresponding orthogonal
pieces and using the finite-sum inequality yields \eqref{eq:upperstrip}, since
$\eps\mu^2\leq c$ keeps $1+\eps(1+\mu)^2$ uniformly bounded.

It remains to treat a possible component $f_<$ supported below frequency $1$.
Let $v_<\in H^1_0(\Om)$ solve $\cL_0v_<=\cL_\eps f_<$.  Spectral calculus and
constant-coefficient regularity give
\begin{equation}\label{eq:lowfrequencyelliptic}
 \norm[L^2(\Om)]{\cL_\eps f_<}+\norm[H^2(\Om)]{v_<}
 \leq C\norm[L^2(\Om)]{f_<}.
\end{equation}
By the corrector estimate \eqref{eq:H1corrector}, the Lipschitz bound for
$\Phi_\eps$, and the normal-coordinate thin-collar estimate,
\[
 \int_{\Om_{R\eps}}|\nabla f_<|^2
 \leq C\eps^2\norm[L^2]{\cL_\eps f_<}^2
       +C\int_{\Om_{R\eps}}|\nabla v_<|^2
 \leq C_R\eps\norm[L^2]{f_<}^2,
\]
where the last inequality follows from \eqref{eq:lowfrequencyelliptic} and
the standard $H^2$ collar bound for $v_<$.  Combining this component with 
high-frequency pieces proves the corollary, including packets that straddle
frequency $1$.
\end{proof}

\section{A lower bound for spectral packets}\label{sec:lowerpacket}

\rev{We first prove a local boundary estimate for equations with a small
inhomogeneous term.}  For a $C^2$ function $\psi:\R^{d-1}\to\R$, set
\begin{align*}
 Z(\psi,r)&=\{(x',x_d):|x'|<r,\ \psi(x')<x_d<\psi(x')+r\},\\
 I(\psi,r)&=\{(x',\psi(x')):|x'|<r\}.
\end{align*}
\rev{Define the half-cylinder}
\[
 Q_r=\{(y',y_d):|y'|<r,\ 0<y_d<r\}
\]
and the flattening map
\begin{equation}\label{eq:flatteningmap}
 \Psi_\psi(y',y_d)=(y',y_d+\psi(y')),
 \qquad J_\psi=D\Psi_\psi.
\end{equation}
Notice that $\det J_\psi=1$.  Thus, if $\widetilde w=w\circ\Psi_\psi$,
the principal coefficient in the flattened equation is
\begin{equation}\label{eq:flattenedcoefficient}
 \widetilde B=J_\psi^{-1}(B\circ\Psi_\psi)J_\psi^{-T}.
\end{equation}

\rev{The following lemma bounds the boundary gradient from below in terms of
the interior gradient.  It allows a small inhomogeneous term, as required
for spectral packets.}

\begin{lemma}
\label[lemma]{lem:microcauchy}
Fix $\kappa\in(0,1)$ and $c_1,C_1>0$. There exists $\eta,c>0$, depending
only on $d,\kappa,c_1,C_1$, with the following property.  Suppose that $B$ is
real symmetric and satisfies
$\kappa|\xi|^2\leq B(x)\xi\cdot\xi\leq\kappa^{-1}|\xi|^2$, with
$\norm[L^\infty]{\nabla B}\leq C_1$.  Suppose also that
$\psi(0)=0$, $\nabla\psi(0)=0$, and
$\norm[L^\infty]{\nabla^2\psi}\leq C_1$.  Let
$w\in H^1(Z(\psi,2);\C)$, $q\in L^2(Z(\psi,2);\C)$, and $\beta\in\R$ satisfy
\begin{equation}\label{eq:localpde}
 -\operatorname{div}(B\nabla w)=\beta w+q\quad\text{in }Z(\psi,2),
 \qquad w=0\quad\text{on }I(\psi,2),
\end{equation}
with $|\beta|\leq C_1$, and assume that
\begin{align}
 \int_{Z(\psi,1)}|\nabla w|^2
     &\geq c_1\int_{Z(\psi,2)}|\nabla w|^2,\label{eq:massratio}\\
 \norm[L^2(Z(\psi,2))]{q}
     &\leq\eta\norm[L^2(Z(\psi,2))]{\nabla w},\label{eq:qsmall}
\end{align}
then
\begin{equation}\label{eq:localcauchy}
 \int_{I(\psi,1)}|\nabla w|^2\dd\sigma
 \geq c\int_{Z(\psi,2)}|\nabla w|^2\dd x.
\end{equation}
The boundary gradient trace is well defined by the local boundary $H^2$
regularity for Lipschitz coefficients and $L^2$ right-hand sides.
\end{lemma}

\begin{proof}
\rev{The lemma extends \cite[Lemma~5.3]{KenigLinShen2013} to equations with a
small inhomogeneous term.  For complex-valued functions, we regard
$(\Re w,\Im w)$ as a two-component real system with diagonal principal part.}

Suppose that the conclusion is false.  Then there are admissible
$B_n,\psi_n,\beta_n,w_n,q_n$ for which \eqref{eq:localcauchy} fails with
$c=1/n$ and \eqref{eq:qsmall} holds with $\eta=1/n$. By homogeneity, we
normalize
\begin{equation}\label{eq:micronormalization}
                     \int_{Z(\psi_n,2)}|\nabla w_n|^2=1.
\end{equation}
Thus $\norm[L^2]{q_n}\to0$ and
\begin{equation}\label{eq:microboundaryzero}
             \int_{I(\psi_n,1)}|\nabla w_n|^2\dd\sigma\longrightarrow0.
\end{equation}

The zero trace on the bottom face yields the uniform Poincar\'e's inequality
\begin{equation}\label{eq:uniformpoincare}
 \norm[L^2(Z(\psi_n,2))]{w_n}
 \leq C\norm[L^2(Z(\psi_n,2))]{\nabla w_n}.
\end{equation}
Indeed, this follows first on $Q_2$ by integrating along vertical lines and
then on $Z(\psi_n,2)$ through \eqref{eq:flatteningmap}; the constant is
controlled by the uniform $C^{1,1}$ graph bound.

Set $\widetilde w_n=w_n\circ\Psi_{\psi_n}$ and define
$\widetilde B_n$ by \eqref{eq:flattenedcoefficient}.  Since the maps and their
inverses have uniformly bounded $C^{1,1}$ norms, the matrices
$\widetilde B_n$ are uniformly elliptic and uniformly Lipschitz on $Q_2$.
The flattened equation is exactly
\begin{equation}\label{eq:flattenedpde}
 -\operatorname{div}_y(\widetilde B_n\nabla_y\widetilde w_n)
   =\beta_n\widetilde w_n+\widetilde q_n,
 \qquad \widetilde q_n=q_n\circ\Psi_{\psi_n},
\end{equation}
because $\det J_{\psi_n}=1$.  Its bottom trace is zero.

The bounds on $\psi_n$ and the normalization at the origin yield, after a
subsequence, $\psi_n\to\psi$ in $C^1(\overline{B'_2})$ and weak-* in
$W^{2,\infty}$.  Arzel\`a--Ascoli and the boundedness of $\beta_n$ likewise give
\begin{equation}\label{eq:microcompactness}
 \widetilde B_n\longrightarrow\widetilde B\quad\text{uniformly on }Q_2,
 \qquad \beta_n\longrightarrow\beta,
 \qquad \widetilde w_n\rightharpoonup\widetilde w\quad\text{in }H^1(Q_2).
\end{equation}
Here and in the following, we pass to further subsequences without relabeling.  The local boundary $H^2$ estimate for divergence-form equations with Lipschitz
coefficients and $L^2$ forcing gives
\begin{equation}\label{eq:microh2}
 \norm[H^2(Q_{3/2})]{\widetilde w_n}
 \leq C\big(\norm[H^1(Q_{7/4})]{\widetilde w_n}
            +\norm[L^2(Q_{7/4})]{\widetilde q_n}\big)\leq C;
\end{equation}
See, for example, \cite[Section~6.3.2, Theorem~4]{Evans2010}, applied after
localization on the flat boundary; the difference-quotient proof uses only the
uniform $W^{1,\infty}$ bound on the coefficients.  After passing to a further
subsequence, we obtain a weak convergence in $H^2(Q_{3/2})$.  Rellich
compactness then gives
\begin{equation}\label{eq:microstrongh1}
             \widetilde w_n\longrightarrow\widetilde w
             \quad\text{strongly in }H^1(Q_{4/3}).
\end{equation}
Changes of variables in \eqref{eq:massratio} and
\eqref{eq:micronormalization}, together with the uniform equivalence of the
flattened and physical gradients, therefore, imply
\begin{equation}\label{eq:micrononzero}
                  \int_{Q_1}|\nabla\widetilde w|^2\dd y\geq c_2>0,
\end{equation}
where $c_2$ depends only on the constants in the statement.

We may pass to the limit in \eqref{eq:flattenedpde} to obtain
\begin{equation}\label{eq:microlimitequation}
 -\operatorname{div}(\widetilde B\nabla\widetilde w)=\beta\widetilde w
 \quad\text{in }Q_{4/3},
 \qquad \widetilde w=0\quad\text{on }\{y_d=0\}.
\end{equation}
The trace operator applied to the first derivatives maps $H^2(Q_{3/2})$
boundedly into $H^{1/2}(\{|y'|<4/3,y_d=0\})$, and this space embeds compactly
into $L^2$ on the smaller face. Thus, the weak $H^2$ convergence, the strong
$H^1$ convergence on $Q_{4/3}$, the convergence of
$J_{\psi_n}$, and \eqref{eq:microboundaryzero} imply
\begin{equation}\label{eq:microzerocauchydata}
 \nabla\widetilde w=0\quad\text{on }\{|y'|<1,y_d=0\};
\end{equation}
in particular, the conormal trace
$\widetilde B\nabla\widetilde w\cdot e_d$ vanishes there.

The uniform Lipschitz bounds and the uniform convergence in
\eqref{eq:microcompactness} show that $\widetilde B$ is itself Lipschitz.
Choose a smaller ball that meets the flat face and extends $\widetilde B$ evenly
across $y_d=0$; the extension remains elliptic and Lipschitz.  Extend
$\widetilde w$ by zero to the lower half.  To verify the equation across the
interface, integrate \eqref{eq:microlimitequation} by parts against an
arbitrary compactly supported test function.  The only interface term is the
conormal trace, which vanishes by \eqref{eq:microzerocauchydata}; hence the
zero extension is a weak solution of
$-\operatorname{div}(\widetilde B\nabla\widetilde w)=\beta\widetilde w$ across
the face.  It vanishes in the open lower half-ball. The unique continuation for
elliptic equations with Lipschitz principal coefficients
\cite{AronszajnKrzywickiSzarski1962} now forces it to vanish in the whole
ball, and the interior unique continuation propagates the vanishing through
$Q_1$. This contradicts \eqref{eq:micrononzero}.  Every estimate used above
is uniform over the admissible class, which proves the result.
\end{proof}

Fix a uniform $C^2$ boundary atlas.  After a rigid motion sending
$P\in\pa\Om$ to the origin and the inward normal to $e_d$, write the
corresponding patch as $Z(\psi_P,r)$ and denote its image in the original
coordinates by $D(P,r)$; set $\Delta(P,r)=D(P,r)\cap\pa\Om$.

After translating $P$ to the origin, rotating the inward normal to $e_d$, and
rescaling by $\eps$, the oscillatory coefficient and the boundary graph satisfy
the uniform hypotheses of
\cref{lem:microcauchy}.  Rescaling its conclusion gives the form used on
physical boundary patches.

\begin{corollary}
\label[corollary]{cor:rescaledcauchy}
Fix $c_1>0$ and $L>0$.  There exist $\eps_1,\eta,c>0$, depending only on
$d,\kappa,M,\Om,c_1$, and $L$, such that the following holds.  Let
$0<\eps<\eps_1$, $P\in\pa\Om$, $\mu\geq0$, and suppose
$f\in H^1(D(P,2\eps);\C)$ and $q\in L^2(D(P,2\eps);\C)$ satisfy
\begin{equation}\label{eq:rescaledpde}
 (\cL_\eps-\mu^2)f=q\quad\text{in }D(P,2\eps),\qquad
 f=0\quad\text{on }\Delta(P,2\eps),
\end{equation}
together with
\begin{align}
 \int_{D(P,\eps)}|\nabla f|^2
   &\geq c_1\int_{D(P,2\eps)}|\nabla f|^2,\label{eq:rescaledmass}\\
 \eps^2\int_{D(P,2\eps)}|q|^2
   &\leq\eta^2\int_{D(P,2\eps)}|\nabla f|^2,
       \label{eq:rescaledforcing}\\
 \eps^2\mu^2&\leq L.\label{eq:rescaledzerothorder}
\end{align}
Then
\begin{equation}\label{eq:rescaledcauchy}
 \int_{\Delta(P,\eps)}|\nabla f|^2\dd\sigma
 \geq \frac c\eps\int_{D(P,2\eps)}|\nabla f|^2\dd x.
\end{equation}
\end{corollary}

\begin{proof}
Let $O_P\in O(d)$ be the orthogonal matrix provided by the boundary chart.
Set $x=P+\eps O_Pz$ and set $w(z)=f(P+\eps O_Pz)$.  The rescaled equation is
\begin{equation}\label{eq:explicitrescaling}
 -\operatorname{div}_z\!\left(B_{P,\eps}(z)\nabla_z w\right)
 =\eps^2\mu^2w+\eps^2q(P+\eps O_Pz),
\end{equation}
where
\[
 B_{P,\eps}(z)=O_P^{T}A(P/\eps+O_Pz)O_P,
 \qquad
 \psi_{P,\eps}(z')=\eps^{-1}\psi_P(\eps z').
\]
The matrices $B_{P,\eps}$ remain uniformly elliptic, and their Lipschitz
seminorms are bounded by a constant that depends only on $M$.  The rescaled
graphs have a uniform $C^2$ bound because
\[
 \norm[L^\infty]{\nabla^2\psi_{P,\eps}}
 \leq\eps\norm[L^\infty]{\nabla^2\psi_P}.
\]
Taking $\beta=\eps^2\mu^2$, the zeroth-order coefficient in
\eqref{eq:explicitrescaling} is bounded by $L$ in view of
\eqref{eq:rescaledzerothorder}.  The mass-ratio condition is
unchanged by scaling, since both gradient integrals acquire the factor
$\eps^{d-2}$.  For the forcing term,
\[
 \frac{\norm[L^2]{\eps^2q(P+\eps O_P\,\cdot)}^2}
      {\norm[L^2]{\nabla_z w}^2}
 =\frac{\eps^4\eps^{-d}\int|q|^2}
        {\eps^2\eps^{-d}\int|\nabla_x f|^2}
 =\frac{\eps^2\int|q|^2}{\int|\nabla_x f|^2}
 \leq\eta^2.
\]
Thus, all hypotheses of \cref{lem:microcauchy} hold with uniform constants,
and its conclusion applies on the unit-scale patch.  Finally, surface and
volume measures scale as $\eps^{d-1}$ and $\eps^d$, respectively.  Their
ratio produces exactly the factor $\eps^{-1}$ in
\eqref{eq:rescaledcauchy}.
\end{proof}

\rev{The following lemma estimates the packet energy in a microscopic boundary
layer.  Its lower bound will be used in the rescaled Cauchy estimate.}

\begin{lemma}
\label[lemma]{lem:layermass}
There exist $\eps_0,\delta_0,\gamma_0,\rho,c,C>0$ such that, if
$0<\eps<\eps_0$, $0<\gamma\leq\gamma_0$, $\eps\mu^2\leq\delta_0$, and
$f\in\mathscr P_\eps(\mu,\gamma)$, then
\begin{equation}\label{eq:layertwosided}
 c\eps\mu^2\norm[L^2(\Om)]{f}^2
 \leq\int_{\Om_{\rho\eps}}|\nabla f|^2\dd x
 \leq C\eps\mu^2\norm[L^2(\Om)]{f}^2.
\end{equation}
\end{lemma}

\begin{proof}
Since the coefficients are real, the spectral projection defining
$\mathscr P_\eps(\mu,\gamma)$ commutes with complex conjugation.  It is
therefore enough to prove the estimate for real-valued $f$: applying that
estimate to $\operatorname{Re}f$ and $\operatorname{Im}f$ and adding it gives
the complex-valued statement.

By ellipticity and Poincar\'e's inequality, the first Dirichlet frequency of
$S_\eps$ has a uniform positive lower bound
$s_*>0$.  We include $\gamma_0\leq\min\{1,s_*/4\}$ in the choice of constants;
therefore, every nonzero packet has $\mu\geq3s_*/4$ and $\mu\geq2\gamma$.  Thus,
after decreasing $\delta_0$ below the constant in
\cref{cor:upperstrip}, estimate \eqref{eq:upperstrip} gives the upper bound in
\eqref{eq:layertwosided}.  We prove the lower bound and normalize
$\norm[L^2]{f}=1$; \eqref{eq:packetbasic} applies. Let
\[
 F=\cL_\eps f=\mu^2f+q,
 \qquad \norm[L^2]{q}\leq C\gamma\mu,
\]
and let $v\in H^1_0(\Om)$ solve $\cL_0v=F$.  Elliptic estimates and
\eqref{eq:packetbasic} give
\begin{equation}\label{eq:vestimates}
 \norm[H^1]{v}\leq C\mu,\qquad
 \norm[H^2]{v}\leq C\mu^2.
\end{equation}
\rev{For the first estimate, note that}
\[
 \norm[H^{-1}]{F}=\norm[H^{-1}]{\cL_\eps f}
 \leq C\norm[L^2]{\nabla f}\leq C\mu,
\]
and Lax--Milgram gives $\norm[H^1_0]{v}\leq C\mu$.  The boundary $H^2$
estimate for the constant-coefficient operator gives the second bound of
$\norm[L^2]{F}\leq C\mu^2$.
By \cref{prop:corrector},
\begin{equation}\label{eq:correctorpacket}
 \norm[L^2]{\nabla f-(\nabla\Phi_\eps)\nabla v}
 \leq C\eps\mu^2.
\end{equation}
\rev{Indeed, define the corrector error by}
\[
 e_\eps=f-v-(\Phi_\eps-x)\cdot\nabla v\rev{.}
\]
\rev{The corrector estimate and
\eqref{eq:Phibound}, \eqref{eq:vestimates} give}
\[
 \norm[L^2]{\nabla f-(\nabla\Phi_\eps)\nabla v}
 \leq \norm[L^2]{\nabla e_\eps}
      +\norm[L^\infty]{\Phi_\eps-x}\norm[L^2]{\nabla^2v}
 \leq C\eps\mu^2.
\]

We claim
\begin{equation}\label{eq:vboundarylower}
 \int_{\pa\Om}|\nabla v|^2\dd\sigma
 \geq c\mu^2-C\eps\mu^4-C\gamma\mu^2.
\end{equation}
\rev{Translate the origin so that $|x|$ is bounded on $\overline\Om$ by a
structural constant, and use $x$ as the Rellich vector field.}  Multiplying
$-\operatorname{div}(\widehat A\nabla v)=F$ by $x\cdot\nabla v$ and
integrating twice by parts gives
\begin{align}
 \mathscr R(v)&:=\int_{\pa\Om}(x\cdot n)
       \widehat A\nabla v\cdot\nabla v\dd\sigma
 \notag\\
 &= (2-d)\int_\Om(\mu^2f+q)v\dd x
      -2\int_\Om(\mu^2f+q)x\cdot\nabla v\dd x.\label{eq:packetRellich}
\end{align}
There are no additional tangential boundary terms because $v=0$ on
$\pa\Om$, hence $\nabla v=(\partial_n v)n$ there.  To verify the signs and the
main term explicitly, use $f=v=0$ on $\pa\Om$ and
$\norm[L^2]{f}=1$.  Writing $v=f+(v-f)$ and integrating the second term by
parts gives
\begin{align}
 2\int_\Om f\,x\cdot\nabla v
 &=-d-2d\int_\Om f(v-f)
      -2\int_\Om(v-f)x\cdot\nabla f\notag\\
 &=d-2d\int_\Om fv
      -2\int_\Om(v-f)x\cdot\nabla f.\label{eq:packetRellichAlgebra}
\end{align}
 Substituting \eqref{eq:packetRellichAlgebra} into
\eqref{eq:packetRellich} shows that the part of $\mathscr R(v)$ that
contains $\mu^2$ is
\begin{equation}\label{eq:packetMainTerm}
 2\mu^2+(d+2)\mu^2\int_\Om f(v-f)
       +2\mu^2\int_\Om(v-f)x\cdot\nabla f.
\end{equation}
The corrector error $e_\eps$ belongs to $H^1_0(\Om)$, and
\cref{prop:corrector} gives
\begin{equation}\label{eq:correctorerror}
 \norm[H^1_0]{e_\eps}\leq C\eps\mu^2,
 \qquad \norm[L^2]{f-v}\leq C\eps\mu^2.
\end{equation}
Indeed, the second estimate follows from
\[
 \norm[L^2]{f-v}
 \leq\norm[L^2]{e_\eps}
    +\norm[L^\infty]{\Phi_\eps-x}\norm[L^2]{\nabla v}
 \leq C\eps\mu^2+C\eps\mu\leq C\eps\mu^2,
\]
where the last inequality uses the uniform positive lower bound for $\mu$.
Moreover,
\[
 \left|\int_\Om e_\eps x\cdot\nabla f\right|
 \leq C\norm[L^2]{f}\norm[H^1]{e_\eps},
\]
by integration by parts, while
\[
 \left|\int_\Om(\Phi_\eps-x)\cdot\nabla v\,x\cdot\nabla f\right|
 \leq C\eps\norm[L^2]{\nabla v}\norm[L^2]{\nabla f}
 \leq C\eps\mu^2.
\]
It follows from \eqref{eq:packetMainTerm}--\eqref{eq:correctorerror} that all
non-main terms containing $\mu^2$ are bounded by $C\eps\mu^4$.  The terms in
\eqref{eq:packetRellich} that contain $q$ are bounded by
\[
 C\norm[L^2]{q}\big(\norm[L^2]{v}+\norm[L^2]{\nabla v}\big)
 \leq C\gamma\mu^2.
\]
Consequently
\[
 \mathscr R(v)\geq2\mu^2-C\eps\mu^4-C\gamma\mu^2.
\]
Finally, $|\mathscr R(v)|\leq C\int_{\pa\Om}|\nabla v|^2$ by the boundedness of
$\Om$ and ellipticity.  This proves \eqref{eq:vboundarylower}.

Choose $\rho>0$ below both the tubular neighborhood radius of $\pa\Om$ and
the thickness in \eqref{eq:Phiinvertible}.  \rev{If $g\in H^1(0,\rho\eps)$, averaging
$g(0)=g(t)-\int_0^t g'(s)\dd s$ in $t$ gives}
\[
 |g(0)|^2\leq \frac C\eps\int_0^{\rho\eps}|g(t)|^2\dd t
              +C\eps\int_0^{\rho\eps}|g'(t)|^2\dd t.
\]
Applying this componentwise to $g(t)=\nabla v(y-tn(y))$ and using the
uniformly bounded Jacobian of boundary normal coordinates yields
\begin{equation}\label{eq:normalFTC}
 \int_{\pa\Om}|\nabla v|^2\dd\sigma
 \leq \frac C\eps\int_{\Om_{\rho\eps}}|\nabla v|^2\dd x
      +C\eps\int_{\Om_{\rho\eps}}|\nabla^2v|^2\dd x.
\end{equation}
On $\Om_{\rho\eps}$, \eqref{eq:Phiinvertible} and
\eqref{eq:correctorpacket} imply
\[
 \int_{\Om_{\rho\eps}}|\nabla v|^2
 \leq C\int_{\Om_{\rho\eps}}|(\nabla\Phi_\eps)\nabla v|^2
 \leq C\int_{\Om_{\rho\eps}}|\nabla f|^2+C\eps^2\mu^4.
\]
Rearranging this inequality and then using
\eqref{eq:vboundarylower}, \eqref{eq:vestimates}, and
\eqref{eq:normalFTC}, we find
\begin{align*}
 \frac1\eps\int_{\Om_{\rho\eps}}|\nabla f|^2
 &\geq \frac c\eps\int_{\Om_{\rho\eps}}|\nabla v|^2-C\eps\mu^4\\
 &\geq c\mu^2-C\eps\mu^4-C\gamma\mu^2.
\end{align*}
First choose $\gamma_0$ so that $C\gamma_0\leq c/4$ and then choose
$\delta_0$ so that $C\delta_0\leq c/4$.  Since
$\eps\mu^2\leq\delta_0$, the two error terms are thus absorbed.
This proves the lower bound.
\end{proof}

\rev{The following covering lemma gives uniform bounds on the volume and
boundary overlaps.}

\begin{lemma}\label[lemma]{lem:boundarycover}
There exist $\rho_c>0$, $\eps_c>0$, and an integer $N_{\mathrm{ov}}\geq1$,
depending only on the fixed boundary atlas, such that for every
$0<\eps<\eps_c$ one can choose
 a finite number $P_j\in\pa\Om$ with
\begin{align}
 \Om_{\rho_c\eps}&\subset\bigcup_j D(P_j,\eps),
       \label{eq:coverstrip}\\
 \sum_j\mathbf1_{D(P_j,2\eps)}&\leq N_{\mathrm{ov}}\quad\text{in }\Om,
       \label{eq:covervolumeoverlap}\\
 \sum_j\mathbf1_{\Delta(P_j,\eps)}&\leq N_{\mathrm{ov}}\quad\text{on }\pa\Om.
       \label{eq:coverboundaryoverlap}
\end{align}
The same volume-overlap bound holds after all radii are multiplied by any
fixed structural constant.
\end{lemma}

\begin{proof}
Take a maximal $a_{\mathrm{sep}}\eps$-separated set on the compact manifold
$\pa\Om$, where $a_{\mathrm{sep}}>0$ is smaller than the bi-Lipschitz
constants of the fixed atlas.  Maximality
gives a tangential cover of $\pa\Om$.  The tubular neighborhood projection
$x\mapsto\pi(x)\in\pa\Om$ is well defined for
$\operatorname{dist}(x,\pa\Om)<\rho_c\eps$ when $\rho_c$ is sufficiently
small; the chart centered near $\pi(x)$ then contains $x$, proving
\eqref{eq:coverstrip}.  If a point belongs to $D(P_j,2\eps)$, all corresponding
centers lie in a boundary ball of radius $C\eps$.  The disjoint boundary balls
of radius $a_{\mathrm{sep}}\eps/2$ around those centers have comparable surface
measure, so a standard packing argument bounds their number by
$N_{\mathrm{ov}}$.  This proves
\eqref{eq:covervolumeoverlap}; the same argument on the boundary proves
\eqref{eq:coverboundaryoverlap} and also the fixed-dilation assertion.
\end{proof}

We can now transfer the strip mass to the boundary.  A good--bad patch
decomposition combines the preceding collar lower bound with the local Cauchy
estimate while preserving summability of the packet remainder.

\begin{theorem}
\label{thm:packet}
There exist $\eps_0,\delta,\gamma,c,C>0$ such that, for every $0<\eps<\eps_0$,
$\mu>0$, and $f\in\mathscr P_\eps(\mu,\gamma)$ satisfying
\begin{equation}\label{eq:packetrange}
                         \eps\mu^2\leq\delta,
\end{equation}
one has
\begin{equation}\label{eq:packettwosided}
 c\mu^2\norm[L^2(\Om)]{f}^2
 \leq\norm[L^2(\pa\Om)]{\nabla f}^2
 \leq C\mu^2\norm[L^2(\Om)]{f}^2.
\end{equation}
\end{theorem}

\begin{proof}
Let $s_*>0$ denote the uniform lower bound for the first Dirichlet frequency,
and initially impose $\gamma<s_*/4$.  The final value of
$\gamma$ will be fixed below, after $K_{\mathrm{loc}}$ and
$\eta(K_{\mathrm{loc}})$. Then every nonzero packet has 
$\mu\geq3s_*/4$ and $\mu\geq2\gamma$. Thus, the low-frequency part of
\cref{prop:upperpacket} is also bounded by $C\mu^2\norm[L^2]{f}^2$, and the
high-frequency hypothesis follows from
$\eps^2\mu^2=\eps(\eps\mu^2)\leq\eps_0\delta$; moreover, $1+\eps\mu^2\leq1+\delta$.  After decreasing $\eps_0$ and $\delta$ once, the
upper estimate follows.  For the lower estimate, \eqref{eq:packetbasic}
applies.  Normalize
$\norm[L^2]{f}=1$ and use the notation
$q=(\cL_\eps-\mu^2)f$.  Apply the proof of \cref{lem:layermass}, replacing its
collar constant by $\min\{\rho,\rho_c\}$ if necessary.  The normal-collar
argument is unchanged; only its structural constants vary.  Choose the family
furnished by \cref{lem:boundarycover}.  Fix a large number $K_{\mathrm{loc}}$
and call an index $j$ good if
\begin{align}
 \int_{D(P_j,2\eps)}|\nabla f|^2
   &\leq K_{\mathrm{loc}}\int_{D(P_j,\eps)}|\nabla f|^2,\label{eq:good1}\\
 \eps^2\int_{D(P_j,2\eps)}|q|^2
   &\leq\eta^2\int_{D(P_j,2\eps)}|\nabla f|^2.\label{eq:good2}
\end{align}
Here, $\eta=\eta(K_{\mathrm{loc}})$ is supplied by \cref{lem:microcauchy}, with
$c_1=K_{\mathrm{loc}}^{-1}$ and $\beta=\eps^2\mu^2\leq1$.  All remaining
bounds in that lemma are already controlled by the structural
coefficient and atlas constants.

For indices that fail \eqref{eq:good1}, bounded overlap and
\cref{cor:upperstrip} (used with a fixed structural dilation $R_0$) give
\[
 \sum_{j\text{ failing }\eqref{eq:good1}}
       \int_{D(P_j,\eps)}|\nabla f|^2
 \leq \frac C{K_{\mathrm{loc}}}\int_{\Om_{R_0\eps}}|\nabla f|^2
 \leq \frac C{K_{\mathrm{loc}}}\eps\mu^2.
\]
Similarly, for indices that fail \eqref{eq:good2},
\[
 \sum_{j\text{ failing }\eqref{eq:good2}}
       \int_{D(P_j,\eps)}|\nabla f|^2
 \leq C\eta^{-2}\eps^2\norm[L^2]{q}^2
 \leq C\eta^{-2}\eps^2\gamma^2\mu^2.
\]
We fix the constants in the following order. First, take $K_{\mathrm{loc}}$ so
large that $C/K_{\mathrm{loc}}$ is smaller than one quarter of the layer-mass
constant. Next, take $\eta=\eta(K_{\mathrm{loc}})$ from
\cref{lem:microcauchy}, choose
$0<\gamma\leq\min\{\gamma_0,s_*/4\}$ as required by
\cref{lem:layermass}, and then reduce $\eps_0$ so that it lies below the
corrector, atlas, and covering thresholds and so that
$C\eta^{-2}\eps_0\gamma^2$ is smaller than another quarter of that constant.
Finally, choose
$\delta\leq\delta_0$ small enough for \cref{lem:layermass} and the upper
packet estimate, and also so that
$\eps^2\mu^2=\eps(\eps\mu^2)\leq\eps_0\delta\leq1$. Thus, the rescaled
zeroth-order coefficient in \eqref{eq:explicitrescaling} belongs to the fixed
class used when $\eta$ was selected.  Equations \eqref{eq:coverstrip} and
\eqref{eq:layertwosided} then show that the good cylinders retain at least
$c\eps\mu^2$ of gradient mass. In every good cylinder,
\cref{cor:rescaledcauchy} applies.  Summing \eqref{eq:rescaledcauchy} and using
\eqref{eq:coverboundaryoverlap} gives
\[
 N_{\mathrm{ov}}\int_{\pa\Om}|\nabla f|^2\dd\sigma
 \geq\sum_{j\ \mathrm{good}}
       \int_{\Delta(P_j,\eps)}|\nabla f|^2\dd\sigma
 \geq\frac c\eps\sum_{j\ \mathrm{good}}
                \int_{D(P_j,\eps)}|\nabla f|^2\dd x
 \geq c\mu^2.
\]
This proves the lower estimate and hence the theorem.
\end{proof}

\begin{remark}\label[remark]{rem:whynew}
For a single eigenfunction with an eigenvalue greater than $1$, the lower estimate
in \eqref{eq:packettwosided} is \cite[Theorem~1.4]{KenigLinShen2013}; the
bounded-frequency case is included in the above argument.  The point of
\cref{lem:microcauchy,thm:packet} is the stability under the packet remainder
$q$, whose norm is $O(\gamma\mu)$.  A single-eigenfunction estimate is not, by
itself, sufficient for multidimensional wave observability because
eigenfrequencies need not have a uniform gap.
\end{remark}

\section{From packets to wave observability}\label{sec:observability}

\rev{Let $S$ be a positive self-adjoint operator on a Hilbert space $H$, and set}
\[
 \mathcal D_{\mathrm{fin}}(S)
 =\bigcup_{\Lambda>0}\Ran\one_{[0,\Lambda]}(S).
\]
In the applications below, $S$ has a compact resolvent, so
$\mathcal D_{\mathrm{fin}}(S)$ consists of finite spectral sums.  We first
convert an upper estimate on each unit
packet into a uniform time-integrated observation estimate.

\begin{lemma}
\label[lemma]{lem:admissibility}
Let $H$ and $Y$ be Hilbert spaces, let $S$ be a positive self-adjoint
operator on $H$ with compact resolvent and
$s_*:=\inf\Spec(S)>0$, and let
$\mathcal C:\mathcal D_{\mathrm{fin}}(S)\to Y$ be linear.  Assume that there is
a constant $C_{\mathrm{up}}>0$ such that, for every integer $k\geq0$ and every
$h\in H$,
\begin{equation}\label{eq:abstractupperpacket}
 \norm[Y]{\mathcal C\,\one_{[k,k+1)}(S)h}
 \leq C_{\mathrm{up}}(1+k)\norm[H]{\one_{[k,k+1)}(S)h},
 \qquad k=0,1,\ldots.
\end{equation}
Then $\mathcal C$ extends to a bounded operator from $D(S^2)$, equipped with
its graph norm, into $Y$.  Moreover, for every $T>0$ there is $K_T>0$,
depending only on $T,C_{\mathrm{up}},s_*$, such that
\begin{equation}\label{eq:schradmissibility}
 \int_0^T\norm[Y]{\mathcal C e^{itS}h}^2\dd t
 \leq K_T\norm[H]{Sh}^2,
 \qquad h\in D(S).
\end{equation}
The left-hand side is defined by completion from finite spectral sums.
Consequently, $z\mapsto\mathcal C z$ is admissible for $z''+S^2z=0$ on the
energy space $D(S)\times H$, with a constant that depends only on
$T,C_{\mathrm{up}},s_*$.  If
\eqref{eq:schradmissibility} holds on $[0,T_a]$ with constant $K_{T_a}$, then
one may take
\begin{equation}\label{eq:admissibilitygrowth}
              K_T\leq\left\lceil\frac{T}{T_a}\right\rceil K_{T_a},
              \qquad T>0.
\end{equation}
\end{lemma}

\begin{proof}
We first take $h$ to be a finite spectral sum.  The resulting bound then
defines the observation for every $h\in D(S)$ by completion.  Moreover, with
$h_k=\one_{[k,k+1)}(S)h$,
\eqref{eq:abstractupperpacket} and weighted Cauchy--Schwarz give
\[
 \norm[Y]{\mathcal C h}
 \leq C_{\mathrm{up}}\sum_k(1+k)\norm[H]{h_k}
 \leq C C_{\mathrm{up}}
       \left(\sum_k(1+k)^4\norm[H]{h_k}^2\right)^{1/2}
 \leq C_{s_*}C_{\mathrm{up}}\norm[H]{S^2h}.
\]
Thus $\mathcal C$ extends continuously from $D(S^2)$, equipped with its graph
norm, to $Y$.  Choose
$\chi\in C_c^\infty(\R)$ equal to $1$ on $[0,T]$, and decompose
$h=\sum_{k\geq0}h_k$ into unit spectral intervals.  Using the unitary Fourier
transform in $t$, Plancherel's theorem gives
\[
 \int_\R\norm[Y]{\chi(t)\mathcal C e^{itS}h}^2\dd t
 =\int_\R\norm[Y]{\sum_k \mathcal C\,\widehat\chi(\tau-S)h_k}^2\dd\tau.
\]
The spectral multiplier $\widehat\chi(\tau-S)$ leaves each unit packet
invariant.  Hence \eqref{eq:abstractupperpacket} bounds the norm of its $k$th
contribution by
\[
 C_{\mathrm{up}}(1+k)a(\tau-k)\norm[H]{h_k},\qquad
 a(t)=\sup_{0\leq s\leq1}|\widehat\chi(t-s)|.
\]
The function $a$ is nonnegative and rapidly decreasing.  Put
$b_k=(1+k)\norm[H]{h_k}$. Pointwise Cauchy--Schwarz gives
\[
 \left(\sum_k a(\tau-k)b_k\right)^2
 \leq\left(\sum_k a(\tau-k)\right)
      \left(\sum_k a(\tau-k)b_k^2\right).
\]
The first factor is uniformly bounded in $\tau$, while integration of the
second gives $\norm[L^1(\R)]{a}\sum_k b_k^2$.  The Plancherel integral is
therefore bounded by
\[
 C_{\chi}C_{\mathrm{up}}^2\sum_k(1+k)^2\norm[H]{h_k}^2
 \leq C_{\chi,s_*}C_{\mathrm{up}}^2\norm[H]{Sh}^2.
\]
This proves \eqref{eq:schradmissibility}.  Finally, for
$z(t)=\cos(tS)z_0+S^{-1}\sin(tS)z_1$, set
\[
 h_+=\tfrac12(z_0-iS^{-1}z_1),\qquad
 h_-=\tfrac12(z_0+iS^{-1}z_1).
\]
Then $z(t)=e^{itS}h_++e^{-itS}h_-$, and
\eqref{eq:schradmissibility} applied to the two terms, with time reversed for
$h_-$, proves the wave admissibility with
$\norm{Sh_+}^2+\norm{Sh_-}^2=\tfrac12(\norm{Sz_0}^2+\norm{z_1}^2)$.
The final assertion follows by partitioning $[0,T]$ into at most
$\lceil T/T_a\rceil$ intervals of length at most $T_a$ and translating the
solution on each interval.
\end{proof}

\begin{blackbox}
Let $\mathcal H$ and $Y$ be Hilbert spaces, let $\mathsf A$ be self-adjoint
on $\mathcal H$, and let $\mathsf B:D(\mathsf A)\to Y$ be an admissible
observation.  Fix a packet half-width $\omega>0$, and let $K_{\pi/\omega}$ be
such that
\begin{equation*}\tag{M1}\label{eq:Milleradmissibility}
 \int_0^{\pi/\omega}\norm[Y]{\mathsf B e^{it\mathsf A}W}^2\dd t
 \leq K_{\pi/\omega}\norm[\mathcal H]{W}^2,
 \qquad W\in D(\mathsf A).
\end{equation*}
Assume that, for some $D_{\mathrm{pkt}}>0$, every spectral packet of width
$2\omega$ satisfies
\begin{equation*}\tag{M2}\label{eq:Millerpacket}
 \norm[\mathcal H]{W}^2\leq D_{\mathrm{pkt}}\norm[Y]{\mathsf B W}^2,
 \quad
 W\in\Ran\one_{[\lambda-\omega,\lambda+\omega]}(\mathsf A),
 \quad \lambda\in\R.
\end{equation*}
Then there exist $T_*,C_*>0$, depending only on
$\omega,D_{\mathrm{pkt}},K_{\pi/\omega}$, such that
\begin{equation*}\tag{M3}\label{eq:Millerobservability}
 \norm[\mathcal H]{W}^2
 \leq C_*\int_0^{T_*}\norm[Y]{\mathsf B e^{it\mathsf A}W}^2\dd t,
 \qquad W\in\mathcal H.
\end{equation*}
This is the sufficient direction of \cite[Corollary~2.17]{Miller2012}.
In particular, $T_*$ and $C_*$ are independent of spectral multiplicities,
finite-dimensional truncations, and their largest eigenvalues.
\end{blackbox}

\begin{proposition}
\label[proposition]{prop:wavepacketcriterion}
Let $H,Y,S$, and $\mathcal C$ be as in \cref{lem:admissibility}.  Assume the
upper packet estimate \eqref{eq:abstractupperpacket}, and suppose that there are
fixed constants $\gamma>0$ and $C_{\mathrm{low}}>0$ such that the reverse
estimate
\begin{equation}\label{eq:abstractlowerpacket}
 \norm[H]{Sh}^2\leq C_{\mathrm{low}}\norm[Y]{\mathcal C h}^2
 \quad\text{for every }h\in\Ran\one_{[\mu-\gamma,\mu+\gamma]}(S)
 \text{ and every }\mu\in\R.
\end{equation}
Fix also a reference time $T_a>0$, and choose $K_a>0$ so that
the admissibility estimate
\begin{equation}\label{eq:referenceadmissibility}
 \int_0^{T_a}\norm[Y]{\mathcal C e^{itS}h}^2\dd t
 \leq K_a\norm[H]{Sh}^2,
 \qquad h\in D(S),
\end{equation}
holds.  Then there exist an observation time $T_*>0$ and a constant $C_*>0$
with the following property: for every initial datum
$(z_0,z_1)\in D(S)\times H$, the energy solution
\begin{equation}\label{eq:abstractwavesolution}
 z(t)=\cos(tS)z_0+S^{-1}\sin(tS)z_1
\end{equation}
of $z''+S^2z=0$ satisfies
\begin{equation}\label{eq:abstractwaveobs}
 \norm[H]{Sz_0}^2+\norm[H]{z_1}^2
 \leq C_*\int_0^{T_*}\norm[Y]{\mathcal C z(t)}^2\dd t.
\end{equation}
The constants $T_*$ and $C_*$ depend only on
$\gamma,C_{\mathrm{low}},s_*,T_a$, and $K_a$; in particular, they are
independent of
spectral multiplicities, the dimension of a spectral truncation and its
largest eigenvalue.
\end{proposition}

\begin{proof}
\rev{We verify admissibility and the packet lower bound for the first-order
wave system.}

\smallskip
\noindent\emph{Step 1: energy space and diagonalization.}
Set $X=D(S)$ with the norm $\norm[X]{h}=\norm[H]{Sh}$.  Since $S\geq s_*>0$,
this norm is equivalent to the graph norm. In $X$, the operator $S$ is
self-adjoint with the domain $D(S^2)$: in the spectral representation, the scalar
inner product on $X$ has weight $s^2$, while $S$ remains multiplied by $s$.
Define
\[
 U(z_0,z_1)
 =2^{-1/2}(z_0-iS^{-1}z_1,\ z_0+iS^{-1}z_1)
 =:(w_+,w_-).
\]
The identity
\[
 \norm[X\oplus X]{U(z_0,z_1)}^2
 =\norm[H]{Sz_0}^2+\norm[H]{z_1}^2
\]
shows that $U$ is unitary from the wave energy space $D(S)\times H$ to
$X\oplus X$.  If $z$ solves $z''+S^2z=0$, then
\[
 U(z(t),z'(t))=(e^{itS}w_+,e^{-itS}w_-)
               =e^{it\mathcal G}(w_+,w_-),
 \qquad \mathcal G=S\oplus(-S).
\]
Thus, $\mathcal G$ is the self-adjoint first-order generator to which the black
box will be applied.

\smallskip
\noindent\emph{Step 2: the transformed observation and admissibility.}
Since
\[
 z(t)=2^{-1/2}\big(e^{itS}w_++e^{-itS}w_-\big),
\]
the transformed observation is
\[
 \mathcal B(w_+,w_-)=2^{-1/2}\mathcal C(w_++w_-).
\]
The preceding lemma gives $\mathcal C:D(S^2)\to Y$ and wave admissibility;
equivalently, $\mathcal B$ is an admissible observation for
$e^{it\mathcal G}$ on
$X\oplus X$.  Moreover, its admissibility constant on any interval of length
$T$ is at most $\lceil T/T_a\rceil K_a$ when $T\geq T_a$, and at most $K_a$
when $T\leq T_a$.

\smallskip
\noindent\emph{Step 3: verification of the lower bound of the packet.}
Put $\widetilde\gamma=\min\{\gamma,s_*/4\}$.  The estimate
\eqref{eq:abstractlowerpacket} is inherited by packets of this smaller width.
Because the positive and negative spectra of $\mathcal G$ are separated by
$2s_*$, a $\widetilde\gamma$-packet of $\mathcal G$ reaches at most one branch. Therefore, this packet consists of either vectors $(h,0)$, with $h$ in a
packet of $S$, or vectors $(0,h)$, with $h$ in the corresponding reflected
packet.  In the first case,
\[
 \norm[X\oplus X]{(h,0)}^2=\norm[H]{Sh}^2
 \leq C_{\mathrm{low}}\norm[Y]{\mathcal C h}^2
 =2C_{\mathrm{low}}\norm[Y]{\mathcal B(h,0)}^2,
\]
and the negative branch is identical.  Hence \eqref{eq:Millerpacket} holds for
$(\mathcal G,\mathcal B)$ with packet constant
$D_{\mathrm{pkt}}=2C_{\mathrm{low}}$.

\smallskip
\noindent\emph{Step 4: application of the black box.}
By Step~3, we may apply the black box to
$(\mathsf A,\mathsf B)=(\mathcal G,\mathcal B)$ with
$\omega=\widetilde\gamma$ and $D_{\mathrm{pkt}}=2C_{\mathrm{low}}$.
The constant in \eqref{eq:Milleradmissibility} is
$K_{\pi/\widetilde\gamma}$, which is bounded in terms of
$T_a,K_a$, and $\widetilde\gamma$ by Step~2.  Hence
\eqref{eq:Millerobservability} gives $T_*$ and $C_*$ depending only on
$\gamma,C_{\mathrm{low}},s_*,T_a,K_a$.  Using the identities in Steps~1--2,
its conclusion
is exactly \eqref{eq:abstractwaveobs}.  Neither the spectral dimension nor the
largest eigenvalue enters either constant.
\end{proof}

\begin{proof}[Proof of \cref{thm:main}]
Fix $N\leq\delta\eps^{-1}$ and restrict $S_\eps$ to
$\cA_{\eps,N}$.  If this space is zero, the assertion is immediate.
Otherwise, the ellipticity and Poincar\'e's inequality give a positive
lower bound for $\inf\Spec S_\eps$, uniform in $\eps$.

For a unit-width packet centered at $\mu\leq\sqrt N+1$, we have
\[
 \eps\mu^2\leq2\eps N+2\eps\leq2\delta+2\eps_0,
 \qquad
 \eps^2\mu^2\leq2\delta\eps+2\eps^2.
\]
After decreasing $\eps_0$ and $\delta$ once more,
\cref{prop:upperpacket} thus
gives
\[
 \norm[L^2(\pa\Om)]{\nabla h}^2
 \leq C(1+\mu)^2\norm[L^2(\Om)]{h}^2
\]
for every unit packet below $\sqrt N$, uniformly in $\eps$ and $N$.  Taking
square roots and centering $[k,k+1)$ at $k+\tfrac12$ gives
\eqref{eq:abstractupperpacket}.  Hence \cref{lem:admissibility} gives a
uniform upper observation estimate on every fixed interval.

Choose the fixed packet width $\gamma$ from \cref{thm:packet}.  If a nonzero
$h\in\Ran\one_{[\mu-\gamma,\mu+\gamma]}(S_\eps)$ lies in the truncated
space, then $\mu\leq\sqrt N+\gamma$ and hence
$\eps\mu^2\leq C\eps N+C\eps\gamma^2$.  Since the packet is nonzero and
$\gamma<s_*/4$, also $\mu\geq s_*/2>0$.  Viewed in the full space, the
restricted packet is a spectral subpacket of the same interval for $S_\eps$.
First reduce $\eps_0$ so that $C\eps_0\gamma^2$ is below the packet threshold,
and then reduce $\delta$.  Hence \cref{thm:packet}, together with
$\norm{S_\eps h}\leq(\mu+\gamma)\norm{h}\leq C\mu\norm{h}$, yields
\[
 \norm[L^2]{S_\eps h}^2\leq C\norm[L^2(\pa\Om)]{\nabla h}^2.
\]
This is \eqref{eq:abstractlowerpacket}, with constants independent of
$\eps$ and $N$.  We may therefore apply \cref{prop:wavepacketcriterion} with
$H=L^2(\Om)$, $Y=L^2(\pa\Om;\C^d)$, and
$\mathcal C h=\nabla h|_{\pa\Om}$.  Together with the admissibility estimate
above, this
gives both sides of the observation inequality on a fixed interval of length
$T_*$.  Set $T_0=2T_*$.  For $T\geq T_0$, the intervals
$[jT_*,(j+1)T_*]$, $0\leq j<\lfloor T/T_*\rfloor$, lie in $[0,T]$.
Applying the lower estimate on each time translate and using conservation of
energy gives
\[
 \int_0^T\norm[Y]{\mathcal C u(t)}^2\dd t
 \geq cT_*\Big\lfloor\frac{T}{T_*}\Big\rfloor
       \cE_\eps(u_{\eps,0},u_{\eps,1})
 \geq cT\cE_\eps(u_{\eps,0},u_{\eps,1}).
\]
For the upper estimate, cover $[0,T]$ by at most
$\lceil T/T_*\rceil$ intervals of length $T_*$ and use the fixed-interval
admissibility estimate on every translate.  Since $T\geq2T_*$, this gives the
upper bound $CT\cE_\eps(u_{\eps,0},u_{\eps,1})$.  This proves
\eqref{eq:observability} with constants independent of $T,\eps$, and $N$.
\end{proof}

\begin{proof}[Proof of \cref{cor:control}]
For Dirichlet solutions, the conormal derivative
$\pa_{\nu_\eps}u=n\cdot A(x/\eps)\nabla u$ is uniformly comparable with
$|\nabla u|$ on $\pa\Om$, because the tangential derivative vanishes and
\eqref{eq:ellipticity} holds.  More precisely,
$\nabla u=(\partial_n u)n$ and
\[
 \kappa|\nabla u|\leq|\pa_{\nu_\eps}u|
 \leq\kappa^{-1}|\nabla u|\quad\text{on }\pa\Om.
\]
Thus \cref{thm:main} is equivalent, with uniform constants, to the
corresponding two-sided estimate with $\pa_{\nu_\eps}u$ as the observation.

If $\cA_{\eps,N}=\{0\}$, take $g_\eps=0$.  Both the terminal constraint and
\eqref{eq:controlnorm} are then immediate.  \rev{Assume now that $\cA_{\eps,N}\ne\{0\}$ and set}
\[
 E_{\eps,N}=\cA_{\eps,N}\times\cA_{\eps,N},\qquad
 \norm[E_{\eps,N}]{(p_0,p_1)}^2
 =\norm[L^2]{S_\eps p_0}^2+\norm[L^2]{p_1}^2.
\]
For $p=(p_0,p_1)\in E_{\eps,N}$, denote the associated homogeneous Dirichlet
wave by $p(t)$ and define
\[
 \mathcal O_Tp=\pa_{\nu_\eps}p
       \in L^2(\pa\Om\times(0,T)).
\]
The two bounds in \eqref{eq:observability} are equivalent to
\begin{equation}\label{eq:HUMobservationbounds}
 cT\norm[E_{\eps,N}]{p}^2
 \leq\norm[L^2(\pa\Om\times(0,T))]{\mathcal O_Tp}^2
 \leq CT\norm[E_{\eps,N}]{p}^2.
\end{equation}
Consequently, the HUM Gramian
$G_T=\mathcal O_T^*\mathcal O_T$ satisfies
$cT\Id\leq G_T\leq CT\Id$ on $E_{\eps,N}$.

Only the low-frequency part of the initial state enters the constraint.  It
defines the continuous functional
\begin{equation}\label{eq:HUMfunctional}
 \ell_\theta(\zeta_0,\zeta_1)
 =\langle P_{\eps,N}\theta_{\eps,1},\zeta_0\rangle_{H^{-1},H^1_0}
  -(P_{\eps,N}\theta_{\eps,0},\zeta_1)_{L^2}
 \quad\text{on }E_{\eps,N}.
\end{equation}
By Lax--Milgram there is a unique $p\in E_{\eps,N}$ such that
\begin{equation}\label{eq:HUMvariational}
 (\mathcal O_Tp,\mathcal O_T\zeta)_{L^2(\pa\Om\times(0,T))}
 =\ell_\theta(\zeta)\qquad(\zeta\in E_{\eps,N}).
\end{equation}
Take $g_\eps=\mathcal O_Tp$ (changing its sign merely changes the convention
in Green's identity).  For the transposition solution with this boundary
value, integration by parts against every adjoint wave $\zeta$ gives
\[
 [\langle\pa_tv_\eps,\zeta\rangle_{H^{-1},H^1_0}
   -(v_\eps,\pa_t\zeta)_{L^2}]_{0}^{T}
 =-\int_0^T\!\int_{\pa\Om}g_\eps\pa_{\nu_\eps}\zeta\dd\sigma\dd t.
\]
Together with \eqref{eq:HUMfunctional}--\eqref{eq:HUMvariational}, this identity
shows that the symplectic pairing of the projected terminal state with every
terminal adjoint state vanishes.  Since the homogeneous adjoint wave group is
an isomorphism of $E_{\eps,N}$, its terminal states fill that space.  Hence both
low-frequency components of the terminal state vanish.  The same identity
also serves as the definition of the transposition solution for
$L^2$ boundary data and initial data in $L^2\times H^{-1}$.  Its well-posedness
follows from the standard energy-duality construction in HUM~\cite{Lions1988}.

\rev{Define}
\[
 \norm[H^{-1}_\eps]{F}=\norm[L^2]{S_\eps^{-1}F}.
\]
The spectral and variational definitions give
\[
 \norm[H^{-1}_\eps]{F}^2
 =\langle F,\cL_\eps^{-1}F\rangle
 =\sup_{0\ne z\in H^1_0(\Om)}
   \frac{|\langle F,z\rangle|^2}
        {\int_\Om A(x/\eps)\nabla z\cdot\overline{\nabla z}\dd x}.
\]
The form bounds in \eqref{eq:ellipticity} therefore give
\begin{equation}\label{eq:uniformnegativenorm}
 c\norm[H^{-1}(\Om)]{F}
 \leq\norm[H^{-1}_\eps]{F}
 \leq C\norm[H^{-1}(\Om)]{F},
\end{equation}
with constants independent of $\eps$.  Moreover $P_{\eps,N}$ is contractive
both in $L^2$ and in $\norm[H^{-1}_\eps]{\cdot}$, because it commutes with
$S_\eps$.  It follows that
\begin{equation}\label{eq:HUMdualnorm}
 \norm[E_{\eps,N}^*]{\ell_\theta}^2
 =\norm[L^2]{P_{\eps,N}\theta_{\eps,0}}^2
  +\norm[H^{-1}_\eps]{P_{\eps,N}\theta_{\eps,1}}^2,
\end{equation}
up to the harmless choice of the real or complex duality convention.

Finally, let $x_\theta\in E_{\eps,N}$ be the Riesz representative of
$\ell_\theta$.  Then $p=G_T^{-1}x_\theta$, and the minimal control satisfies
\begin{equation}\label{eq:HUMminimalnorm}
 \norm[L^2(\pa\Om\times(0,T))]{g_\eps}^2
 =\ip{G_T^{-1}x_\theta}{x_\theta}_{E_{\eps,N}}.
\end{equation}
Indeed, every other feasible control differs from $g_\eps$ by an element
orthogonal to $\Ran\mathcal O_T$ and therefore cannot have a smaller norm.
Equations \eqref{eq:HUMobservationbounds}, \eqref{eq:uniformnegativenorm}, and
\eqref{eq:HUMdualnorm} prove \eqref{eq:controlnorm}.  They also show that both
constants there are comparable to $T^{-1/2}$ for $T\geq T_0$.  One can also see the related implementation in
\cite[Section~4]{LinShen2022}.
\end{proof}

\section*{Acknowledgments}

During the preparation of this manuscript, the authors used ChatGPT (OpenAI) to assist with the development and presentation of the mathematical arguments and the drafting and revision of the text. The authors have reviewed and verified the arguments and results and take full responsibility for the accuracy and integrity of the manuscript.

\end{document}